\documentclass[11pt, reqno]{amsart}
\usepackage[a4paper, left=1in, right=1in, top=1in, bottom=1in]{geometry}
\usepackage{amsmath,hyperref}
\usepackage{amssymb}
\usepackage[T1]{fontenc}
\usepackage{newtxtext}
\usepackage{newtxmath}

\usepackage{hyperref}
\hypersetup{colorlinks,
	linkcolor={black},
	filecolor={black},
	citecolor={blue},
	urlcolor={blue}}  

\usepackage{amsthm}
\newtheorem{thm}{Theorem}[section]

\newtheorem{lma}{Lemma}[section]
\numberwithin{equation}{section}
\begin{document}
	
	\title[Congruence Results on $t-$Schur  overpartitions]{Extending Recent Congruence Results on $t-$Schur  overpartitions}
	\author{K. C. Ajeyakashi, H. S. Sumanth Bharadwaj and S. Chandankumar}
	\address[K. C. Ajeyakashi]{Department of Mathematics, Faculty of Engineering and Technology, Jain (Deemed-to-be University), Bangalore, Karnataka, India.} \email{k.ajeya@jainuniversity.ac.in}
	
	\address[K. C. Ajeyakashi] {Department of Mathematics and Statistics, Faculty of Natural Sciences, M. S. Ramaiah University of Applied Sciences, Peenya Campus, Peenya 4th Phase, Bangalore-560 058, Karnataka, India.} \email{19MPRP727001@msruas.ac.in}
	
	\address[H. S. Sumanth Bharadwaj]{Department of Mathematics and Statistics, M. S. Ramaiah University of Applied Sciences, Peenya Campus, Peenya 4th Phase, Bengaluru-560 058, Karnataka, India.} 
	\email{sumanthbharadwaj@gmail.com}
	
	\address[S. Chandankumar]{Department of Mathematics and Statistics, M. S. Ramaiah University of Applied Sciences, Peenya Campus, Peenya 4th Phase, Bengaluru-560 058, Karnataka, India.} 
	\email{chandan.s17@gmail.com}
	
	\subjclass[2010]{11P83, 05A15, 05A17} \keywords{$t-$Schur  overpartitions, color partitions}
	
	\maketitle
	\begin{abstract}
		Recently, Nadji and Ahmia~\cite{nadji2021} introduced the notion of $t$-Schur overpartitions and investigated their combinatorial and arithmetic properties. In this paper, we extend their work and establish several new congruence relations for $t$-Schur overpartitions. For example, for all $n \ge 0$ we prove
		\[
		\overline{S_9}(24n+23)\equiv 0 \pmod{32}.
		\]
	\end{abstract}
	
	\section{Introduction}
	An overpartition of a nonnegative integer $n$ is a nonincreasing sequence of positive integers whose sum is $n$, where the first occurrence of each distinct part may be overlined. The generating function for overpartitions is
	
	\begin{equation}
		\sum_{n=0}^{\infty }\overline{p}\left ( n\right )q^n=\prod_{n=1}^{\infty }\frac{\left ( 1+q^{n} \right )}{\left ( 1-q^{n} \right )}=\frac{(-q;q)_\infty}{(q;q)_\infty}=1+2q+4q^{2}+8q^{3}+14q^{4}+24q^{5}+...,
	\end{equation}
	where $\overline{p}(n)$ denotes the number of overpartitions of $n$. Here
	
	\[
	(a;q)_n :=
	\begin{cases}
		1, & n=0,\\[2pt]
		\displaystyle\prod_{k=1}^{n} (1 - a q^{k-1}), & n>0,
	\end{cases}
	\]
	is the $q$-shifted factorial and
	\[(a;q)_\infty=\lim_{n\rightarrow\infty}(a;q)_n, \qquad |q|<1.\]
	For more details on arithmetic properties of the overpartition function one can refer to \cite{2ABH,2ABH1,2ABH2,2ABE,2KB}.
	Recently, Nadji and Ahmia~\cite{nadji2021} investigated partitions into parts that are 
	simultaneously regular and distinct from both combinatorial and arithmetic perspectives. 
	For any odd integer $t \geq 3$, they defined the $t$-Schur partitions of $n$ as those 
	partitions into parts congruent to $i \pmod{2t}$, where $i\in I(t)$ and
	\[
	I(t) = \{1, 3, 5, 7, \ldots, 2t-1\} \setminus \{t\}.
	\]
	In particular, when $t=3$, this reduces to the classical Schur partitions. These partitions may also be interpreted as \emph{biregular partitions}, namely partitions in which 
	no part is divisible by both $\ell_1$ and $\ell_2$ for $\gcd(\ell_1,\ell_2)=1$. 
	In the case of $t$-Schur partitions, the parameters are $(\ell_1,\ell_2)=(2,t)$, so the classical 
	Schur partitions correspond to the $(2,3)$-biregular case. The number of $t$-Schur partitions, denoted by $S_t(n)$, the generating function is given by
	\[
	\sum_{n \geq 0} S_t(n) q^n: 
	= \prod_{n \geq 1} \frac{1-q^{t(2n-1)}}{1-q^{2n-1}}
	= \frac{f_2 f_t}{f_1 f_{2t}},\,\,\text{where}\,\, f_{k}:=\left ( q^{k} ;q^{k}\right )_{\infty }.
	\]
	Moreover, $S_t(n)$ also counts two equivalent classes of partitions: 
	\begin{enumerate}
		\item partitions of $n$ into parts that are simultaneously $2$-regular and $t$-distinct 
		(no part occurs $t$ or more times), and 
		\item partitions of $n$ into distinct parts that are not divisible by $t$.
	\end{enumerate}
	Extending this idea of $t$-Schur partitions, Nadji and Ahmia~\cite{nadji2021} introduced the notion of $t$-Schur overpartitions of a positive integer $n$, denoted by $\overline{S_t}(n)$, which is generalization of $t$-Schur partitions. The generating function is given by
	\begin{equation}\label{tsp}
		\sum_{n \geq 0} \overline{S_t}(n) q^n := \prod_{n \geq 1} \frac{(1 + q^{2n-1})(1 - q^{t(2n-1)})}
		{(1 + q^{t(2n-1)})(1 - q^{2n-1})} = \frac{f_2^3 f_{t}^2 f_{4t}}{f_1^2 f_4 f_{2t}^3}.
	\end{equation}
	For example, $\overline{S_5}(6) = 8$, with the corresponding set:
	\[
	(1^6), (\overline{1},1^5), (3,1^3), (\overline{3},1^3), (3,\overline{1},1^2), (\overline{3},\overline{1},1^2), (3^2), (\overline{3},3).
	\]
	If $\overline{S_t^r}(n)$ denotes the number of $t$-Schur overpartitions of $n$ with $r$-tuples, then
	\begin{equation}\label{tspr}
		\sum_{n \geq 0} \overline{S_t^r}(n) q^n: = \prod_{n \geq 1} \frac{(1 + q^{2n-1})^r(1 - q^{t(2n-1)})^r}
		{(1 + q^{t(2n-1)})^r(1 - q^{2n-1})^r} = \frac{f_2^{3r} f_{t}^{2r} f_{4t}^r}{f_1^{2r} f_4^r f_{2t}^{3r}}.
	\end{equation}
	They also showed that for any odd $t \ge 3$, $\overline{S_t}(n)$ equals the number of partitions of $2n$ into parts not divisible by $t$, where odd parts occur with multiplicity~$2$ and even parts are distinct. They derived Ramanujan-like congruences and dissections for $\overline{S_t}(n)$ in the case $t=3$. In particular, they posed the following directions:
	\begin{enumerate}
		\item What interesting arithmetic properties does $\overline{S_t}(n)$ exhibit?
		\item Are there connections between $t$-Schur partitions/overpartitions and other areas?
	\end{enumerate}
	
	% The authors~\cite{nadji2021} also demonstrated that for any odd $t \geq 3$, the number $\overline{S_t}(n)$ equals	the number of partitions of $2n$ into parts indivisible by $t$, where the odd parts appear only with a multiplicity of 2, and the even parts are distinct. They have also derived several Ramanujan-like congruences and  generating function dissections satisfied by $\overline{S_t}(n)$ for $t = 3$. In the conclusion, the authors highlight a few directions for future exploration:
	
	%    \begin{enumerate}
		%        \item  Are there any connections between $t$-Schur's partitions and overpartitions and other fields or topics?
		%        \item What interesting arithmetic properties do $\overline{S_t}(n)$  exhibit?
		%    \end{enumerate}
	
	For every odd integer $t \ge 3$, $t$-Schur overpartitions are examples of $(2,t)$-biregular overpartitions. In related work, Adiga et al.~\cite{adiga2018congruences} obtained several infinite families of congruences modulo~$8$ for $\overline{S_t}(n)$; for instance, for nonnegative integers $\zeta,n$ and $t\equiv1\pmod{8}$,
	\[
	\overline{S_t}\big(2^{1+\zeta}(16n+14)\big) \equiv 0 \pmod{8}.
	\]
	%     Firstly, note that for every odd integer $ t \geq 3 $, the $t$-Schur overpartitions conform to the definition of $(2, t)$-biregular overpartitions. In \cite{adiga2018congruences} C. Adiga et al. have obtained several infinite families of congruences modulo 8 for $\overline{S_t}(n)$. For example, for any non-negative integers $ \zeta$, $ n$, and $ t \equiv 1 \pmod{8} $,
	% 	\begin{align}
		% \overline{S_t}(n)\left(2^{1 + \zeta}(16n + 14)\right) \equiv 0 \pmod{8}.
		% 	\end{align}
	Later, Shivashankar and Gireesh~\cite{shivashankar2022congruences} established further Ramanujan-like congruences, including infinite families modulo~$3$ and~$8$ for $\overline{S_3}(n)$ and modulo~$8$ for $\overline{S_5}(n)$. We record the following lemma from~\cite{nadji2021}.
	
	% Later, Shivashankar and Gireesh \cite{shivashankar2022congruences} have established several Ramanujan-like congruences and infinite family of congruences modulo 3 and 8 for $\overline{S_3}(n)$ and modulo 8 for $\overline{S_5}(n)$. 
	
	\begin{lma}\label{lem:nadji-ahmia}
		For all $n\ge0$,
		\begin{align}
			% \sum_{n \ge 0} \overline{S}_3(6n + 3) q^n = 2 \frac{f_2^2 f_3^3 f_4 f_{12}}{f_1^5 f_6^2},\\
			\sum_{n \ge 0} \overline{S}_3(12n + 7) q^n &= 8\, \frac{f_2^{4} f_3^4 f_4^4}{f_1^{10} f_6^2},\label{eq2}\\
			\sum_{n \ge 0} \overline{S}_3(12n + 11) q^n &= 16\, \frac{f_2^3 f_3 f_4^4 f_6}{f_1^9}.\label{eq3}
		\end{align}
	\end{lma}
	
	This article extends congruences from~\cite{adiga2018congruences,nadji2021,shivashankar2022congruences} using elementary $q$-series identities, thereby addressing the questions raised in~\cite{nadji2021}. In~\cite{shivashankar2022congruences}, the authors provided $2$-dissection formulas for $\overline{S_3}(n)$. We complement this with the following $3$-dissections for $\overline{S}_{3\ell}(n)$.
	
	%     The result in the following lemma is due to Nadji and Ahmia \cite{nadji2021}.  
	% \begin{lma}For all $n\geq0$, we have
		% 	\begin{align}
			% 	%	\sum_{n \geq 0} \overline{S}_3(6n + 3) q^n = 2 \frac{f_2^2 f_3^3 f_4 f_{12}}{f_1^5 f_6^2},\label{eq1}\\
			% 	&\sum_{n \geq 0} \overline{S}_3(12n + 7) q^n = 8 \frac{f_2^{4} f_3^4 f_4^4}{f_1^{10} f_6^2},\label{eq2}\\
			% 	&\sum_{n \geq 0} \overline{S}_3(12n + 11) q^n = 16 \frac{f_2^3 f_3 f_4^4 f_6}{f_1^9}.\label{eq3}
			% 	\end{align}
		% \end{lma}
	% 	This article further contributes to the topic by extending various congruences obtained in \cite{adiga2018congruences}, \cite{nadji2021}, and \cite{shivashankar2022congruences} using elementary $q$-series identities, thereby addressing the future perspectives outlined in \cite{nadji2021}. In \cite{shivashankar2022congruences} the authors have given the 2-dissection formulas for the generating function $\overline{S_3}(n)$. In the following theorem, we state the 3-dissection formulas for  $\overline{S}_{3\ell}(n)$:
	\begin{thm} \label{tm1}For any positive integer $n$ and $\ell$ is any odd positive integer, we have
		\begin{align}
			&\sum_{n= 0}^{\infty } \overline{S}_{3\ell}(3n)q^n=\frac{f_2^4f_{4}^{} f_3^6f_\ell^2f_{4\ell}^{}}{f_1^8f_{6}^{} f_{12}^{} f_{2\ell}^3} - 4q\frac{f_2^4 f_3^3 f_{12}^2 f_\ell^2 f_{4\ell}^{}}{f_1^7 f_4^{} f_{6}^{}f_{2\ell}^3}, \label{tsp3n}\\
			% &\sum_{n= 0}^{\infty } \overline{S_{3t}}(3n)q^n=\frac{f_2f_{4} f_3^6 }{f_1^6f_{6} f_{12} } - 4q\frac{f_2 f_3^3 f_{12}^2}{f_1^5 f_{6}}, \label{tsp3n}\\
			&\sum_{n= 0}^{\infty } \overline{S}_{3\ell}(3n+1)q^n=
			2\frac{f_2^3 f_3^3 f_{6}^2 f_\ell^2f_{4\ell}^{} }{f_1^7 f_{2\ell}^3f_{12}^{}} 
			- 8q\frac{f_2^3 f_{6}^2 f_{12}^2 f_\ell^2f_{4\ell}^{}}{f_1^6 f_4^{} f_{2\ell}^3},\label{tsp3n1}\\
			%&\sum_{n= 0}^{\infty } \overline{S_{3t}}(3n+1)q^n=
			%2\frac{f_3^3 f_{4}f_{6}^2 }{f_1^5 f_{12}} - 8q\frac{f_{6}^2 f_{12}^2}{f_1^4},\label{tsp3n1}\\
			&\sum_{n= 0}^{\infty } \overline{S}_{3\ell}(3n+2)q^n= 
			4\frac{f_2^2 f_{6}^5 f_\ell^2 f_{4\ell}^{}}{f_1^6 f_{2\ell}^3 f_{12}^{}} 
			- 2\frac{f_2^5 f_3^6 f_{12}^2 f_\ell^2 f_{4\ell}^{}}{f_1^8 f_{6}^4 f_4^{} f_{2\ell}^3}\label{tsp3n2}.
			%        &\sum_{n= 0}^{\infty } \overline{S_{3t}}(3n+2)q^n= 4\frac{ f_{6}^5 f_{4}}{f_1^4 f_2 f_{12}} - 2\frac{f_2^2 f_3^6 f_{12}^2}{f_1^6 f_{6}^4}\label{tsp3n2}.
			%&\overline{S_3}(2^\alpha(6n))\equiv\overline{S_3}(6n)\pmod{4}.\label{tsp4n}
		\end{align}
	\end{thm}
	%\begin{thm}
	%			If \( \zeta \geq 1 \), \( m \geq 0 \), \( n \geq 0 \), and \( \eta \geq 1 \) are any integers, then 
	%			\begin{equation}
		%				\overline{S}_t^{2^\zeta m + \eta}(n) \equiv \overline{S}_t^{\eta}(n) \pmod{2^{\zeta + 1}}. \label{mainth}
		%			\end{equation}
	%\end{thm}
	In Section~\ref{s9}, we establish several Ramanujan-like congruences modulo \(3,8,16,\) and~\(32\) for $\overline{S_9}(n)$. In Section~\ref{s3}, we first prove Theorem~\ref{tm1}, then derive congruences modulo \(3,4,8,12,16,24,\) and~\(32\) for $\overline{S_3}(n)$, and finally relate $\overline{S_3}(n)$ to $\overline{S_3^2}(n)$.
	
	% In Section \ref{s9}, We establish few Ramanujan-like congruences modulo 3, 8, 16, and 32 satisfied by $\overline{S_9}(n)$. We begin  section \ref{s3} by providing a proof of Theorem \ref{tm1}, we derive various Ramanujan-like congruences modulo 3, 4, 8, 12, 16, 24 and 32 satisfied by $\overline{S_3}(n)$ and we relate $\overline{S_3}(n)$ to $\overline{S_3^2}(n)$.
	
	\section{Preliminaries}
	\noindent We begin this section by introducing the Ramanujan’s general theta-function $f(a,b),$ defined as:
	\begin{align}
		f(a,b): = \sum_{n=-\infty}^{\infty} a^{\frac{n(n+1)}{2}} b^{\frac{n(n-1)}{2}}, \quad |ab| < 1.
	\end{align}
	The following definitions of theta-functions  $\varphi$, $\psi$ and $f$ are classical:
	\begin{eqnarray}
		\varphi(q)&:=&f(q,q)= \frac{f_2^5}{f_1^2f_4^2},\\
		\psi(q)&:=&f(q,q^3) =\frac{f_2^2}{f_1},\\
		f(-q)&:=&f(-q,-q^2)=f_1.
	\end{eqnarray}
	\begin{lma}
		Suppose $l \geq 1$, $m \geq 1$ are any integer, and $p$ is any prime. Then we have
		\begin{align}\label{bt}
			f^{p^{l-1}}_{pm}\equiv f_{m}^{p^{l}} \pmod {p^l}.
		\end{align}
	\end{lma}
	\begin{lma} [{\cite[Entry~25]{berndt2012ramanujan}}] The following 2-dissections holds 
		\begin{equation}\label{f12}
			f_1^2 =\frac{f_2 f_8^5}{f_4^2 f_{16}^2}
			-2q\frac{f_2 f_{16}^2}{f_8},
		\end{equation}
		\begin{equation}
			\label{1byf12}
			\frac{1}{f_1^2} =\frac{f_8^5}{f_2^5 f_{16}^2}
			+2q\frac{f_4^2 f_{16}^2}{f_2^5 f_8}.
		\end{equation}
		%	\begin{equation}
			%		\label{1_f1.4}
			%		\frac{1}{f_1^4} =\frac{f_4^{14}}{f_2^{14} f_{8}^4}
			%		+4q\frac{f_4^2 f_{8}^4}{f_2^{10}}.
			%	\end{equation}
	\end{lma}
	
	\begin{lma}[{\cite{xia2012some}}]The following 2-dissection holds:
		\begin{align}
			\frac{f_9}{f_1} &= \frac{f_{12}^3f_{18}}{f_2^2f_{6}f_{36}}+q\frac{f_4^2f_6f_{36}}{f_{2}^3f_{12}}.\label{f9byf1}
		\end{align}
	\end{lma}
	\begin{lma}[{\cite{hirschhorn1993cubic}}]The following 2-dissection hold:
		\begin{align}
			\frac{f_3^3}{f_1} &= \frac{f_4^3f_6^2}{f_2^2f_{12}}+q\frac{f_{12}^3}{f_4},\label{f33byf1}\\
			\frac{f_3}{f_1^3} &= \frac{f_4^6f_6^3}{f_2^9f_{12}^2}+3q\frac{f_4^2f_6f_{12}^2}{f_2^7}.\label{f3byf13}
		\end{align}
	\end{lma}
	\begin{proof}
		Note that Hirschhorn, Garvan and Borwein \cite{hirschhorn1993cubic} proved \eqref{f33byf1}. In the same paper, authors also obtained 
		\begin{equation}\label{gr}
			\frac{f_1^3}{f_3} = \frac{f_4^3}{f_{12}} - 3q\,\frac{f_2^2 f_{12}^3}{f_4^2 f_6^2}.
			\tag{2.11}
		\end{equation}
		Replacing $q$ by $-q$ in \eqref{gr} and using the fact that $(-q;-q)_{\infty} = \dfrac{f_2^3}{f_1 f_4^4}$, we obtain \eqref{f3byf13}.
		
	\end{proof}
	%    \begin{lma}\label{l2}
		%The following $2$-dissections hold.
		%\begin{align}
		%f_1^2&=\dfrac{f_2f_8^5}{f_4^2f_{16}^2}-2q\dfrac{f_2f_{16}^2}{f_8},\label{f12}\\
		%f_1^4&=\dfrac{f_4^{10}}{f_2^2f_{8}^4}-4q\dfrac{f_2^2f_{8}^4}{f_4^2},\label{f14}\\
		%\frac{1}{f_{1}^{2}}&=\frac{f_{8}^{5}}{f_{2}^{5}f_{16}^{2}}+2q\frac{f_{4}^{2}f_{16}^{2}}{f_{2}^{5}f_{8}},\label{1byf12}\\
		%\frac{1}{f_{1}^{4}}&=\frac{f_{4}^{14}}{f_{2}^{14}f_{8}^{4}}+4q\frac{f_{4}^{2}f_{8}^{4}}{f_{2}^{10}}.\label{1byf14}\\
		%\frac{1}{f_1^8} &= \frac{f_4^{28}}{f_2^{28} f_8^8} + 8q \frac{f_4^{16}}{f_2^{24}} + 16q^2 \frac{f_4^4 f_8^8}{f_2^{20}},\label{1byf18}\\
		%\frac{1}{f_1^{12}} &= \frac{f_4^{42}}{f_2^{42} f_8^{12}} + 12q \frac{f_4^{30}}{f_2^{38} f_8^4} + 48q^2 \frac{f_4^{18} f_8^4}{f_2^{34}} + 64q^3 \frac{f_4^6 f_8^{12}}{f_2^{30}}.\label{1byf112}
		%\end{align}
		%\end{lma}
		%Lemma \ref{l2} is a consequence of dissection formulas of Ramanujan, collected from Berndt's book \cite[ p. 40, Entry 25]{BCB}.
		\begin{lma} [{\cite{xia2013analogues}}] The following 2-dissection hold:
			\begin{align}
				\frac{f_3}{f_1} &=
				\frac{f_4 f_6 f_{16}f_{24}^{2}}{f_2^{2} f_8 f_{12}f_{48}}
				+ q \frac{f_6 f_8^2 f_{48}}{f_2^2f_{16}f_{24}},\label{f3f1}\\
				\frac{f_3^2}{f_1^2}
				&=\frac{f_4^4 f_6 f_{12}^2}{f_2^5 f_8 f_{24}}
				+2q\frac{f_4 f_6^2 f_8 f_{24}}{f_2^4 f_{12}},\label{f3.2_f1.2}\\            
				\frac{f_3^4}{f_1^4}
				&=\frac{f_4^8 f_6^2 f_{12}^{4}}{f_2^{10} f_8^2 f_{24}^{2}}
				+ 4q \frac{f_4^5 f_6^3 f_{12}}{f_2^9}
				+ 4q^2 \frac{f_4^2 f_6^4 f_8^2 f_{24}^{2}}{f_2^8 f_{12}^{2}},\label{f34f14}\\
				\frac{1}{f_1 f_3} &= 
				\frac{f_8^2 f_{12}^5}{f_2^2 f_4 f_6^4 f_{24}^2} 
				+ q \frac{f_4^5 f_{24}^2}{f_2^4 f_6^2 f_8^2 f_{12}},\label{1byf1f3}\end{align}
			\begin{align}
				\frac{1}{f_1^2 f_3^2} &= 
				\frac{f_8^5 f_{24}^5}{f_2^5 f_6^5 f_{16}^2 f_{48}^2} 
				+ 2q \frac{f_4^4 f_{12}^4}{f_2^6 f_6^6}+4q^4\frac{f_4^2f_{12}^2 f_{16}^4f_{48}^2}{f_2^5 f_6^5 f_8f_{24}}.\label{1byf12f32}
				%    \frac{1}{f_3^3f_1^3}&=
				%\frac{4q^5 f_4^7 f_{12} f_{16}^4 f_{24} f_{48}^2}{f_2^9 f_6^7 f_8^3}
				%+ \frac{4q^4 f_4 f_8 f_{12}^7 f_{16}^4 f_{48}^2}{f_2^7 f_6^9 f_{24}^3}
				%+ \frac{2q^2 f_4^9 f_{12}^3 f_{24}^2}{f_2^{10} f_6^8 f_8^2}
				%+ \frac{2q f_4^3 f_8^2 f_{12}^9}{f_2^8 f_6^{10} f_{24}^2}\nonumber
				%\\&+ \frac{q f_4^5 f_8^3 f_{24}^7}{f_2^9 f_6^7 f_{12} f_{16}^2 f_{48}^2}
				%+ \frac{f_8^7 f_{12}^5 f_{24}^3}{f_2^7 f_4 f_6^9 f_{16}^2 f_{48}^2}. 
			\end{align}
		\end{lma}
		
		%	\begin{lma}The following $2$-dissection hold
			%		\begin{equation}\label{1byf33byf1}
				%		\frac{f_3^3}{f_1}=\frac{f_4^3f_6^2}{f_2^2f_{12}}+q\frac{f_{12}^3}{f_4}.	
				%		\end{equation}
			%       \end{lma}
		\begin{lma}[{\cite{mdb}}] The following 3-dissection hold:
			\begin{align}
				&\frac{f_2}{f_1^2} = \frac{f_6^4 f_9^6}{f_3^8 f_{18}^3} + 2q \frac{f_6^3 f_9^3}{f_3^7} + 4q^2 \frac{f_6^2 f_{18}^3}{f_3^6},\label{f2byf12}\\
				&\frac{f_1^2}{f_2} = \frac{f_9^2}{f_{18}} - 2q \frac{f_3 f_{18}^2}{f_6 f_9}.\label{f12byf2}
		\end{align}\end{lma}
		\noindent Note that \eqref{f2byf12} is equivalent to the 3-dissection of $1/\varphi(-q)$ and \eqref{f12byf2} is equivalent to the 3-dissection of $\varphi(-q)$.
		\begin{lma}[{ \cite{hirschhorn2014congruence}}] The following 3-dissection holds:
			\begin{align}
				f_1 f_2 = \frac{f_6 f_9^4}{f_3 f_{18}^2} - q f_9 f_{18} - 2q^2 \frac{f_3 f_{18}^4}{f_6 f_9^2}.\label{f1f2}
			\end{align}
		\end{lma}		
		
		\begin{lma} [{\cite[Corollary to Entry~31]{berndt2012ramanujan}}]
			We have 
			\begin{equation}\label{f22byf1}
				\psi(q)=\frac{f_6f_9^2}{f_3f_{18}}+q\frac{f_{18}^2}{f_9}.
			\end{equation}
		\end{lma}
		\begin{lma} [{\cite{toh2012ramanujan}}] The following $3$-dissection holds:
			\begin{align}
				\frac{f_{2}^{3}}{f_{1}^{3}}=\frac{f_{6}}{f_{3}}+3q\frac{f_{6}^{4}f_{9}^{5}}{f_{3}^{8}f_{18}}+6q^{2}\frac{f_{6}^{3}f_{9}^{2}f_{18}^{2}}{f_{3}^{7}}+12q^{3}\frac{f_{6}^{2}f_{18}^{5}}{f_{3}^{6}f_{9}}.\label{f23byf13}
			\end{align}
		\end{lma}	
		\begin{lma} [{\cite{baruah2015partitions}}] The following $3$-dissection holds:
			\begin{align}
				\frac{f_4}{f_1} = \frac{f_{12} f_{18}^4}{f_3^3 f_{36}^2} + q\frac{f_6^2 f_9^3 f_{36}}{f_3^4 f_{18}^2} 
				+ 2q^2\frac{f_6 f_{18} f_{36}}{f_3^3}.\label{f4byf1}
			\end{align}
		\end{lma}
		\begin{lma}[{\cite[Theorem ~2.2]{cui2013}}]\label{a1}
			%(\cite[Theorem 2.2]{cui2013})\label{a1}.
			For any prime $p\ge5$,
			\begin{equation}\label{a17}
				f_1=\sum\limits_{\substack{k=\frac{1-p}{2}\\k\neq\frac{\pm p-1}{6}
				}}^{\frac{p-1}{2}}{(-1)^kq^{\frac{3k^2+k}{2}}f\left(-q^{\frac{3p^2+(6k+1)p}{2}},-q^{\frac{3p^2-(6k+1)p}{2}}\right)}+(-1)^{\frac{\pm p-1}{6}}q^{\frac{p^2-1}{24}}f_{p^2},
			\end{equation}
			where
			\begin{equation}
				\dfrac{\pm p-1}{6}:=\begin{cases}
					\frac{p-1}{6}, & \text{if $p\equiv 1\pmod{6}$},\\
					\frac{-p-1}{6}, & \text{if $p\equiv -1\pmod{6}$}.\nonumber
				\end{cases}
			\end{equation}
		\end{lma}
		
		\begin{lma}[{\cite[Theorem ~2.1]{cui2013}}]
			For any odd prime p,
			\begin{align}
				\psi(q)=\sum_{m=0}^{\frac{p-3}{2}}q^{\frac{m^{2}+m}{2}}f\left ( q^\frac{p^2+(2m+1)p}{2},q^\frac{p^2-(2m+1)p}{2} \right )+q^{\frac{p^{2}-1}{8}}\frac{f_{2p^{2}}^{2}}{f_{p^{2}}}.\label{a2}
			\end{align}
		\end{lma}
		Furthermore,
		\begin{equation*}
			\frac{m^{2}+m}{2}\not\equiv \frac{p^{2}-1}{8}\pmod{p} {\hspace{1mm}} \text{for} {\hspace{1mm}} 0\leq m\leq \frac{p-3}{2}.
		\end{equation*}
		
		\begin{lma}[{\cite[Theorem 2.1]{LW}}]
			For any odd prime $p$,
			\begin{equation}
				f_{1}^{3}=\sum_{k=-\frac{p-1}{2}}^{\frac{p-3}{2}}(-1)^kq^{\frac{k^{2}+k}{2}}B_k(q^p)+(-1)^{\frac{p-1}{2}}pq^{\frac{p^{2}-1}{8}}f_{p^{2}}^{3},\label{abc1}
			\end{equation}
		\end{lma}
		where $$B_k(q):=\frac{1}{2}\sum_{n=-\infty}^{\infty}(-1)^n(2pn+2k+1)q^{\dfrac{pn^2+(2k+1)n}{2}}.$$
		
		%	\begin{lma}\cite[Theorem 2.1]{cui2013} For any odd prime $p$, 
			%		\begin{equation}\label{psi_q}
				%			\psi(q)= \displaystyle \sum_{n=0}^\infty q^{\frac{n(n+1)}{2}}=
				%			\displaystyle \sum_{k=0}^{\frac{p-3}{2}} q^{\frac{k^2+k}{2}}
				%			f \left( q^{\frac{p^2+(2k+1)p}{2}},q^{\frac{p^2-(2k+1)p}{2}}\right)
				%			+ q^{\frac{p^2-1}{8}} \psi(q^{p^2}).
				%		\end{equation}
			%	\end{lma}
		
		%	\begin{lma}\cite[Theorem 2.2]{cui2013} For any prime $p \geq 5$, 
			%		\begin{align}
				%			f(-q)& =f_1=(q,q)_\infty \nonumber \\
				%%			&= \displaystyle \sum_{\substack{k=-\frac{p-1}{2}\\ k\neq \frac{\pm p-1}{6}}}^{\frac{p-1}{2}} (-1)^k q^{\frac{3k^2+k}{2}}
				%			f\left( -q^{\frac{3p^2+(6k+1)p}{2}},-q^{\frac{3p^2-(6k+1)p}{2}}\right)
				%			+ (-1)^{\frac{\pm p^2-1}{6}}q^{\frac{p^2-1}{24}} f(-q^{p^2}).\label{f_q}
				%		\end{align}
			%	\end{lma}
		
		\section{Congruences for $\overline{S_9}(n)$}\label{s9}
		In this section, we establish several congruences and related arithmetic properties of $\overline{S}_9(n)$.
		\begin{thm}
			Let $\zeta, \eta \geq 1$ and $m, n \geq 0$ be integers. For any odd integer $t \geq 3$, we have 
			\begin{equation}
				\overline{S}_t^{2^\zeta m + \eta}(n) \equiv \overline{S}_t^{\eta}(n) \pmod{2^{\zeta + 1}}. \label{mainth}
			\end{equation}
		\end{thm}
		\begin{proof}
			Substituting $r = 2^\zeta m + \eta$ in \eqref{tspr} yields
			\begin{align}
				\sum_{n\geq 0}\overline{S}_t^{2^\zeta m + \eta}(n) q^n 
				&= \frac{f_2^{3\cdot2^\zeta m } f_{t}^{2^{\zeta+1} m } f_{4t}^{2^\zeta m }}{f_1^{2^{\zeta+1} m } f_4^{2^\zeta m} f_{2t}^{3\cdot 2^\zeta m}}\cdot
				\frac{f_2^{3\eta } f_{t}^{2\eta } f_{4t}^{\eta}}{f_1^{2\eta} f_4^{\eta} f_{2t}^{3\eta}}.
			\end{align}
			Thanks to \eqref{bt}, we get 
			\begin{align}
				\sum_{n\geq 0}\overline{S}_t^{2^\zeta m + \eta}(n) q^n \equiv 
				\frac{f_2^{3\eta } f_{t}^{2\eta } f_{4t}^{\eta}}{f_1^{2\eta} f_4^{\eta} f_{2t}^{3\eta}}\pmod{2^{\zeta + 1}}.
			\end{align}
			Comparing coefficients yields the result.
		\end{proof}
		\begin{thm}For any positive integer $n$, we have
			\begin{align}
				% &\overline{S_9}(2n+1)=\frac{f_2f_6f_{18}}{f_1^2f_9^2},\label{res91} \\
				&\sum_{n \geq 0}\overline{S_9}(6n+1)q^n=2\frac{f_2^6f_{3}^4}{f_1^8f_6^2},\label{res92}\\
				&\sum_{n \geq 0}\overline{S_9}(6n+3)q^n=4\frac{f_2^5f_{3}f_6}{f_1^7},\label{res93}\\
				&\sum_{n \geq 0}\overline{S_9}(6n+5)q^n=8\frac{f_2^4f_{6}^4}{f_1^6f_3^2}.\label{res94}
			\end{align}
		\end{thm}
		\begin{proof}
			Substituting $t=9$ in \eqref{tsp}, we have
			\begin{equation}\label{tsp9}
				\sum_{n \geq 0} \overline{S_9}(n) q^n=  \frac{f_2^3 f_{9}^2 f_{36}}{f_1^2 f_4 f_{18}^3}.
			\end{equation}
			Using \eqref{f9byf1} in \eqref{tsp9} and collecting the coefficients of $q^{2n+1}$, we obtain
			
			\begin{equation}\label{tsp91}
				\sum_{n \geq 0} \overline{S_9}(2n+1) q^n = \frac{f_2 f_{6}^2 f_{18}}{f_1^2 f_{9}^2}.
			\end{equation}
			Using \eqref{f2byf12} in \eqref{tsp91} and extracting terms involving the powers $q^{3n}$, $q^{3n+1}$, and $q^{3n+2}$ yields \eqref{res92}, \eqref{res93}, and \eqref{res94}, respectively. 
		\end{proof}
		\begin{thm}For any positive integer $n$ and $\alpha\geq0$, we have
			\begin{align}
				&\overline{S_9}(12n+7)\equiv 0 \pmod{8},\label{res911} \\
				%        &\overline{S_9}(12n+11)\equiv 0 \pmod{16},\label{res912} \\
				&\overline{S_9}(6n+3)\equiv(-1)^{\alpha}\overline{S_9}(3^{\alpha}(6n+3))\pmod{16},\label{res914}\\
				&\overline{S_9}(6n+5)\equiv\overline{S_9}(3^\alpha(6n+5)) \pmod{16},\label{res913} \\
				&\overline{S_9}(24n+23)\equiv 0 \pmod{32}.\label{res9141}         \end{align}
		\end{thm}
		\begin{proof}
			Thanks to \eqref{bt}, \eqref{res92} reduces to 
			\begin{align}
				&\sum_{n \geq 0}\overline{S_9}(6n+1)q^n\equiv 2f_2^2 \pmod{8}\label{res921}.
			\end{align}
			Congruence \eqref{res911} follows from above equation. Now using \eqref{bt} in \eqref{res94}, we get
			\begin{align}
				&\sum_{n \geq 0}\overline{S_9}(6n+5)q^n\equiv 8f_2f_6^3 \pmod{16}.\label{res922}
			\end{align}
			Thanks to \eqref{bt}, \eqref{res93} reduces to
			\begin{align}
				&\sum_{n \geq 0}\overline{S_9}(6n+3)q^n\equiv4\frac{f_2^3f_{3}f_6}{f_1^3}\pmod{16}\label{res931}
			\end{align}
			Using \eqref{f23byf13} in the above equation and extracting terms involving the powers $q^{3n}$, $q^{3n+1}$ and $q^{3n+2}$ respectively, we get
			\begin{align}
				&\sum_{n \geq 0}\overline{S_9}(18n+3)q^n\equiv4f_2^2 \pmod{16},\label{res932}\\
				&\sum_{n \geq 0}\overline{S_9}(18n+9)q^n\equiv12\frac{f_2^5f_{3}^5}{f_1^7f_6}\equiv -4\frac{f_2^3f_{3}f_6}{f_1^3}\pmod{16},\label{res933}\\
				&\sum_{n \geq 0}\overline{S_9}(18n+15)q^n\equiv8f_2f_{6}^3\pmod{16}.\label{res934}
			\end{align}    
			Congruence \eqref{res914} follows from \eqref{res931} and \eqref{res933}. Congruence \eqref{res913} follows from \eqref{res922} and \eqref{res934}.  Using \eqref{bt} in \eqref{res94}, we have
			\begin{align}
				&\sum_{n \geq 0}\overline{S_9}(6n+5)q^n\equiv8\frac{f_2^2f_{6}^4}{f_1^2f_3^2}\pmod{32}.\label{res941}
			\end{align}
			Using \eqref{f3.2_f1.2} and extracting the terms containing $q^{2n+1}$ from both sides of the resulting equation, we get
			\begin{align}
				&\sum_{n \geq 0}\overline{S_9}(12n+11)q^n\equiv16\frac{f_2^4 f_{6}^4}{f_1^4f_3^2}\equiv 16f_4^{}f_6^3\pmod{32}.\label{res9412}
			\end{align}
			Congruence \eqref{res9141} follows from above equation.
		\end{proof}
		
		\begin{thm} \label{mains9}
			Let $p\ge 5$ be a prime and $\left(\frac{-3}{p}\right)=-1$, $\alpha \geq 0$, and $n \geq 0$, we have
			\begin{equation}
				\overline{S_9}\big( 6p^{2\alpha+2}n + (6i + p)p^{2\alpha+1} \big) \equiv 0 \pmod{3},\label{44a1}
			\end{equation}
			where $i\in\{1,2,\dots,p-1\}$.
			
		\end{thm}
		\begin{proof}
			Thanks to \eqref{bt}, \eqref{res92} reduces to
			\begin{align}\label{44a2}
				\sum_{n \geq 0}\overline{S_9}(6n+1)q^n\equiv 2f_1f_3\pmod{3}.
			\end{align}
			Define
			\begin{equation}
				\sum_{n=0}^{\infty} c(n) q^n = f_{1} f_{3}.\label{44a3}
			\end{equation}
			Combining \eqref{44a2} and \eqref{44a3}  we find that
			\begin{equation}
				\overline{S_9}(6n+1) \equiv 2 c(n) \pmod{3},\label{44a4}
			\end{equation}
			Now, we consider the congruence equation
			\begin{equation}
				\frac{3m^{2} + m}{2} + 3 \cdot \frac{3k^{2} + k}{2} \equiv \frac{4p^{2} - 4}{24} \pmod{p},\label{44a5}
			\end{equation}
			which is equivalent to
			\begin{equation}
				(6m + 1)^{2} +3 (6k + 1)^{2} \equiv 0 \pmod{p}.\label{44a6}
			\end{equation}
			where $-\frac{(p-1)}{2} \leq m,k \leq \frac{p-1}{2}$ and $p$ is a prime such that $\left( \frac{-3}{p} \right) = -1.$
			Since $\left( \frac{-3}{p} \right) = -1$, the congruence relation \eqref{44a5} holds if and only if 
			$m, k = \pm \frac{p-1}{6}$. Therefore, if we substitute \eqref{a17} into \eqref{44a3} and then extract the terms in which the powers of $q$ are $pn + \frac{4p^{2} - 4}{24} = pn + \frac{p^2 - 1}{6}$, we arrive at
			\begin{equation}
				\sum_{n=0}^{\infty} c\!\left( pn + \frac{p^2 - 1}{6} \right) q^{pn + \frac{p^2 - 1}{6}}
				= (-1)^{\frac{p-1}{6} + \frac{p-1}{6}}\, q^{\frac{p^2 - 1}{6}} f_{p^2} f_{3p^{2}}.\label{44a7}
			\end{equation}
			Dividing both sides of \eqref{44a7} by $q^{\frac{p^2-1}{6}}$ and then replacing $q^p$ by $q$, we obtain
			\begin{equation}
				\sum_{n=0}^{\infty} c\!\left( pn + \frac{p^2 - 1}{6} \right) q^{n}
				= (-1)^{2\times \frac{p-1}{6}} f_{p} f_{3p},\label{44a8}
			\end{equation}
			which implies that
			\begin{equation}
				\sum_{n=0}^{\infty} c\!\left( p^{2}n + \frac{p^2 - 1}{6} \right) q^{n}
				= f_{1} f_{3}.\label{44a9}
			\end{equation}
			and for $n \geq 0$,
			\begin{equation}
				c\left( p^{2}n + pi + \frac{p^2 - 1}{6} \right) = 0,\label{44a10}
			\end{equation}
			where $i\in\{1,2,\dots,p-1\}$. Combining \eqref{44a3} and \eqref{44a9}, we see that for $n \geq 0$,
			
			\begin{equation}
				c\left( p^{2}n + \frac{p^2 - 1}{6} \right) = c(n).\label{44a11}
			\end{equation}
			Iterating \eqref{44a11} shows that, for all $\alpha \geq 0$,
			\begin{equation}
				c\left( p^{2\alpha}n + \frac{p^{2\alpha} - 1}{6} \right) = c(n).\label{44a12}
			\end{equation}
			Replacing $n$ by $p^2 n + p i + \frac{p^2 - 1}{6}$ in \eqref{44a12} and using \eqref{44a10}, we find that for $n \geq 0$ and $\alpha \geq 0$,
			\begin{equation}
				c\left( p^{2\alpha+2} n + p^{2\alpha+1} i + \frac{p^{2\alpha+2} - 1}{6} \right) = 0.
			\end{equation}
			Again, replacing $n$ by $p^{2\alpha+2}n + p^{2\alpha+1} i + \frac{p^{2\alpha+2} - 1}{6}
			\quad (i = 1, 2, \ldots, p - 1)$
			in \eqref{44a4}, we arrive at \eqref{44a1}.
		\end{proof}
		\begin{thm}
			For any prime $p \equiv 5 \pmod{6}$, $\alpha \geq 0$, and $n \geq 0$, we have
			\begin{equation}
				\overline{S_9}\big( 6p^{2\alpha+2}n + (6i + 3p)p^{2\alpha+1} \big) \equiv 0 \pmod{3},\label{4a1}
			\end{equation}
			where $i\in\{1,2,\dots,p-1\}$.
		\end{thm}
		\begin{proof}  
			Thanks to \eqref{bt}, \eqref{res93} reduces to 
			\begin{align}
				&\sum_{n \geq 0}\overline{S_9}(6n+3)q^n\equiv4\frac{f_2^5f_{3}f_6}{f_1^7}\equiv4\psi(q)\psi(q^3)\pmod{3},\label{res93p}
			\end{align}
			
			The proof is similar to the proof of Theorem \ref{mains9}, where \eqref{res93p} is used.
		\end{proof}
		\begin{thm}
			Let $p\ge 5$ be a prime with $\left(\frac{-18}{p}\right)=-1$, and let $\alpha, n\geq 0$. Then 
			\begin{equation}
				\overline{S}_9^{16m + 1}\left ( 24p^{2\alpha +2} n+\left (24i+ 11p\right)p^{2\alpha +1} \right )\equiv0\pmod{32} ,\label{44b1}
			\end{equation}
			where $i\in\{1,2,\dots,p-1\}$.
		\end{thm}
		\begin{proof}
			Substituting $\zeta=4$, \( \eta= 1 \) and, $t=9$, into \eqref{mainth} and using \eqref{res9412}, we have
			\begin{align}\label{res9412p}
				&\sum_{n \geq 0}\overline{S}_9^{16m + 1}(24n+11)q^n\equiv16f_2f_{3}^3\pmod{32}.
			\end{align}
			The remainder of the proof follows the argument of Theorem~\ref{mains9}, using \eqref{res9412p}.
			
		\end{proof}
		
		\begin{thm}
			Let $p\ge 5$ be a prime with $\left(\frac{-3}{p}\right)=-1$,and let $\alpha, n\geq 0$. Then
			\begin{equation}
				\overline{S}_9^{8m + 1}\left ( 18p^{2\alpha +2} n+3\left (6i+ p\right)p^{2\alpha +1} \right )\equiv0\pmod{16} ,\label{44b1}
			\end{equation}
			where $i\in\{1,2,\dots,p-1\}$.
		\end{thm}
		\begin{proof}
			The proof is similar to the proof of Theorem \ref{mains9}, where \eqref{res932} is used.
		\end{proof}
		
		\begin{thm}
			Let $p\ge 5$ be a prime with $\left(\frac{-9}{p}\right)=-1$, and let $\alpha, n\geq 0$. Then
			\begin{equation}
				\overline{S}_9^{8m + 1}\left ( 12p^{2\alpha +2} n+\left (12i+ 5p\right)p^{2\alpha +1} \right )\equiv0\pmod{16} ,\label{44b1}
			\end{equation}
			where $i\in\{1,2,\dots,p-1\}$.
		\end{thm}
		\begin{proof}  
			Substituting $\zeta=3$, \( \eta= 1 \) and, $t=9$, into \eqref{mainth} and using \eqref{res922}, we have 
			\begin{align}
				&\sum_{n \geq 0}\overline{S}_9^{8m + 1}(12n+5)q^n\equiv 8f_1f_3^3 \pmod{16}.\label{res922p}
			\end{align}
			Remaining part of the proof follows the proof of Theorem \ref{mains9}, where \eqref{res922p} is used.
		\end{proof}
		\begin{thm}
			Let $p\ge 5$ be a prime with $\left(\frac{-3}{p}\right)=-1$, and let $\alpha, n\geq 0$. Then
			\begin{equation}
				\overline{S}_9^{4m + 1}\left ( 6p^{2\alpha +2} n+\left (6i+ p\right)p^{2\alpha +1} \right )\equiv0\pmod{8} ,\label{4s94b1}
			\end{equation}
			where $i\in\{1,2,\dots,p-1\}$.
		\end{thm}
		\begin{proof}  
			The proof is similar to the proof of Theorem \ref{mains9}, where \eqref{res921} is used.
		\end{proof}
		
		\section{Congruences for $\overline{S_3}(n)$}\label{s3}
		In this section, we establish several infinite families of congruences modulo 3, 4, 8, 12, 16, 24 and 32 for $\overline{S_3}(n)$. We begin with the proof of Theorem \ref{tm1}.
		\begin{proof}[Proof of Theorem \ref{tm1}]
			%   Note that it is easy to see that the generating function \eqref{tsp} for $t=3$ can be written as 
			%    \begin{equation}\label{tsp3}
				%		\sum_{n \geq 0} \overline{S_3}(n) q^n = \frac{f_2^3 f_{3}^2 f_{12}}{f_1^2 f_4 f_{6}^3}=\frac{\varphi(q)}{\varphi(-q^2)}\frac{\varphi(-q^3)}{\varphi(-q^6)}=\frac{\varphi(-q^2)}{\varphi(-q)}\frac{\varphi(-q^3)}{\varphi(-q^6)}.
				%\end{equation} 
				Using \eqref{f2byf12} for $\dfrac{f_2}{f_1^2}$ and \eqref{f12byf2} with the substitution $q\mapsto q^2$ for $\dfrac{f_2^2}{f_4}$ in the generating function \eqref{tsp} (with $t=3\ell$, $\ell$ odd), we obtain
				
				%Using \eqref{f12byf2} and \eqref{f2byf12} in the generating function \eqref{tsp} for $t=3\ell$ where $\ell$ is any positive odd integer, we have 
				\begin{equation}\label{tsp31}
					\sum_{n \geq 0} \overline{S}_{3\ell}(n) q^n = \left(\frac{f_6^4 f_9^6}{f_3^8 f_{18}^3} + 2q \frac{f_6^3 f_9^3}{f_3^7} + 4q^2 \frac{f_6^2 f_{18}^3}{f_3^6}\right)\left(\frac{f_{18}^2}{f_{36}} - 2q^2 \frac{f_6 f_{36}^2}{f_{12} f_{18}}\right)\frac{f_{3\ell}^2f_{12\ell}}{f_{6\ell}^3}.
				\end{equation} 
				Equations \eqref{tsp3n}–\eqref{tsp3n2} follow by collecting the coefficients of $q^{3n}$, $q^{3n+1}$, and $q^{3n+2}$, respectively, on both sides of \eqref{tsp31}.
			\end{proof}
			
			\begin{thm}
				For any integers $n\ge 0$, $\alpha\ge 0$, and $\beta\ge 1$, we have
				\begin{align}
					\overline{S_3}\big(3n\big) &\equiv \overline{S_3}(n) \pmod{4},\label{n3n}\\
					\overline{S_3}\big(2^\alpha(3n+j)\big) &\equiv \overline{S_3}(3n+j) \pmod{4}
					&&\text{for each } j\in\{0,1,2\}, \label{sum1}\\
					\overline{S_3}\big(2^{\,2\beta+j}(3n+2)\big) &\equiv \overline{S_3}\big(2^{\,j}(3n+2)\big) \pmod{8}
					&&\text{for } j\in\{0,1\}. \label{tsp4nx0}
				\end{align}
			\end{thm}
			
			\begin{proof}   
				Substituting $\ell=1$ in \eqref{tsp3n} and using \eqref{bt}, we get 
				\begin{align}
					&\sum_{n= 0}^{\infty } \overline{S_3}(3n)q^n\equiv\frac{f_4 f_6}{f_2 f_{12}}\frac{f_3^2}{f_1^2} \pmod{4}.\label{tsp3n3}
				\end{align}
				Using the definition of $ \overline{S_3}(n)$ and above equation, we arrive at \eqref{n3n}. Invoking \eqref{f3.2_f1.2} and extracting the terms involving $q^{2n}$ and $q^{2n+1}$ from both sides of the resulting equation, we have
				\begin{align}
					&\sum_{n= 0}^{\infty } \overline{S_3}(6n)q^n\equiv\frac{f_3^2f_2^5f_6}{f_1^6 f_4f_{12}} \equiv \frac{f_4 f_6}{f_2 f_{12}}\frac{f_3^2}{f_1^2} \pmod{4}\label{tsp3n331}
				\end{align}
				and 
				\begin{align}\label{ref1sp}
					&\sum_{n= 0}^{\infty } \overline{S_3}(6n+3)q^n\equiv2\frac{f_2^2f_3^3f_4f_{12}}{f_1^5 f_6^2}\equiv2\frac{f_2^2f_3^3}{f_1}\pmod{4}.
				\end{align}
				% Thanks to \eqref{bt}, \eqref{tsp3n331} reduces to
				% \begin{align}
					%     &\sum_{n= 0}^{\infty } \overline{S_3}(6n)q^n\equiv\frac{f_3^2f_2^5f_6}{f_1^6 f_4f_{12}}\equiv\frac{f_2^3f_6}{f_4f_{12}}\frac{f_3^2}{f_1^2}\pmod{4}.\label{tsp3n31}
					% \end{align}
				%  Substituting \eqref{f3.2_f1.2} and extracting the terms involving $q^{2n}$ from both sides of the resulting equation, we get
				% \begin{align}
					% &\sum_{n= 0}^{\infty } \overline{S_3}(12n)q^n\equiv\frac{f_2^3f_6}{f_4f_{12}}\frac{f_3^2}{f_1^2}\pmod{4}.\label{tsp3n31}
					%\end{align}
					Combining \eqref{tsp3n3} with \eqref{tsp3n331} and iterating, we obtain
					\begin{align}
						\overline{S_3}(2^\alpha(3n))\equiv\overline{S_3}(3n)\pmod{4}.\label{tsp4n}
					\end{align}
					%Now using \eqref{bt} in equation \eqref{tsp3n321}, we get
					%\begin{align}
					%  &\sum_{n= 0}^{\infty } \overline{S_3}(6n+3)q^n\equiv2\frac{f_2^2f_3^3f_4f_{12}}{f_1^5 f_6^2}\equiv2\frac{f_2^2f_3^3}{f_1}\pmod{4}.\label{tsp3n3212}
					%\end{align}
					Substituting $\ell=1$ in \eqref{tsp3n1} and using \eqref{bt}, we get  
					\begin{align}
						\sum_{n= 0}^{\infty } \overline{S_3}(3n+1)q^n\equiv 2 \frac{f_4f_6^2}{f_2^2f_{12}}\frac{f_3^3}{f_1}\pmod{8}.\label{tsp3n32181}
					\end{align}
					Invoking \eqref{f33byf1} in \eqref{tsp3n32181} and extracting those terms involving the powers $q^{2n}$ and $q^{2n+1}$ from both sides of the resulting equation, we get
					\begin{align}
						\sum_{n= 0}^{\infty } \overline{S_3}(6n+1)q^n\equiv 2 \frac{f_2^4f_3^4}{f_6^2f_{1}^4}\equiv2{f_2^2}\pmod{8}.\label{tsp3n32182}
					\end{align}
					and
					\begin{align}
						\sum_{n= 0}^{\infty } \overline{S_3}(6n+4)q^n\equiv2\frac{f_3^2f_6^2}{f_1^2}\pmod{8}.\label{tsp3n32183}
					\end{align}
					Substituting \eqref{f3.2_f1.2} in the above equation \eqref{tsp3n32183} and extracting those terms involving the powers $q^{2n}$ and $q^{2n+1}$ from both sides of the resulting equation, we have
					\begin{align}
						\sum_{n= 0}^{\infty } \overline{S_3}(12n+4) q^{n}\equiv 2 \frac{f_2^2f_6^2} {f_4f_{12}}\frac{f_3^3}{f_1}\pmod{8}\label{tsp3n321831}
					\end{align}
					and
					\begin{align}
						\sum_{n= 0}^{\infty } \overline{S_3}(12n+10)q^n\equiv 4 \frac{f_3^4f_4f_{12}}{f_2^2f_{6}}\equiv4f_2f_6f_{12}\pmod{8}.\label{tsp3n32184}
					\end{align}
					Applying \eqref{f33byf1} to \eqref{tsp3n321831} and then collecting the $q^{2n+1}$ - terms, we obtain
					\[
					\sum_{n= 0}^{\infty } \overline{S_3}(24n+16)q^n \equiv 2\,\frac{f_1^2 f_3^2 f_6^2}{f_2^2}
					\equiv 2\,\frac{f_3^2 f_6^2}{f_1^2} \pmod{8}. \label{tsp3n321861}
					\]
					% \noindent Invoking \eqref{f33byf1} in \eqref{tsp3n321831} and extracting those terms involving the powers of $q^{2n+1}$ from both sides of the resulting equation, we get
					% \begin{align}
						% \sum_{n= 0}^{\infty } \overline{S_3}(24n+16)q^n\equiv 2\frac{f_1^2f_3^2f_6^2}{f_2^2}\equiv2 \frac{f_6^2f_3^2}{f_1^2}\pmod{8}.\label{tsp3n321861}
						% \end{align}
					Note that from \eqref{tsp3n32183} and \eqref{tsp3n321861}, we get
					\begin{align}
						\overline{S_3}(6n+4)\equiv\overline{S_3}(24n+16)\pmod{8}.\label{tsp3n32186}
					\end{align}
					Substituting $\ell=1$ in \eqref{tsp3n2} and using \eqref{bt}, we have
					\begin{align}
						&\sum_{n= 0}^{\infty } \overline{S_3}(3n+2)q^n\equiv 
						2\frac{f_6^3}{f_{2}}\pmod{4}.\label{tsp3n212}
					\end{align}
					Again, from \eqref{tsp3n212}, we have 
					\begin{align}
						\sum_{n= 0}^{\infty } \overline{S_3}(6n+2)q^n\equiv2\frac{f_3^3}{f_1}  \pmod{4}. \label{tsp3n3218}
					\end{align}
					Note that \eqref{tsp3n1} with $\ell=1$ reduces to 
					\begin{align}
						\sum_{n= 0}^{\infty } \overline{S_3}(3n+1)q^n\equiv2\frac{f_3^3}{f_1}  \pmod{4}. \label{tsp3n3218xo}
					\end{align}
					The following congruence follows from \eqref{tsp3n3218} and \eqref{tsp3n3218xo}, we have
					\begin{align}
						\overline{S_3}(2^{\alpha}(3n+1))\equiv\overline{S_3}(3n+1)\pmod{4}.\label{tsp4nx1}
					\end{align}
					Invoking \eqref{f33byf1} in \eqref{tsp3n3218xo} and extracting those terms involving the powers of $q^{2n}$ and $q^{2n+1}$ respectively from both sides of the resulting equation, we get
					\begin{align}
						\sum_{n= 0}^{\infty } \overline{S_3}(6n+1)q^n\equiv2{f_2^{4}}\equiv2{f_{8}}\pmod{4}. \label{tsp3n3220x}
					\end{align}
					and
					\begin{align}
						\sum_{n= 0}^{\infty } \overline{S_3}(6n+4)q^n\equiv2\frac{f_6^{3}}{f_2}\equiv2\frac{f_6^{2}f_3^{2}}{f_1^2}\pmod{4}. \label{tsp3n3220}
					\end{align}
					From \eqref{tsp3n212} and above equation, we get 
					\begin{align}
						\overline{S_3}(2^{\alpha}(3n+2))\equiv\overline{S_3}(3n+2)\pmod{4}.\label{tsp4nx131c}
					\end{align}
					Congruence \eqref{sum1} follows from \eqref{tsp4n}, \eqref{tsp4nx1} and \eqref{tsp4nx131c}.
					%Combining \eqref{tsp3n31}, \eqref{tsp3n3218xo}, and \eqref{tsp3n3218} yields \eqref{sum1}.
					% Equation \eqref{tsp3n3220} implies
					%\begin{align}
					%\sum_{n= 0}^{\infty } \overline{S_3}(12n+10)q^n\equiv0\pmod{4}. \label{tsp3n32202}
					%\end{align}
					%Remove  $q^{8n},q^{8n+1},...,q^{8n+7}$ from  \eqref{tsp3n3220x}, we get 
					%\begin{align}
					%  \sum_{n= 0}^{\infty } \overline{S_3}(48n+1)q^n\equiv 2f_{1}\pmod{4}. \label{tsp3n3220x1}
					% \end{align}
				%and    
				%\begin{align}
				%     \sum_{n= 0}^{\infty } \overline{S_3}(48n+6i+1)q^n\equiv 0\pmod{4} \quad i={1,2,3,4,5,6,7} \label{tsp3n3220x2} 
				%    \end{align}
			
			\noindent Substituting $\ell=1$ in \eqref{tsp3n2} and using \eqref{bt}, we get        \begin{align}
				\sum_{n= 0}^{\infty } \overline{S_3}(3n+2)q^n\equiv 
				4\frac{ f_{6}^3}{f_2}- 2\frac{f_6^4}{f_1^2 f_{3}^2}\pmod{8}\label{tsp3n2new}.
			\end{align}
			Squaring \eqref{1byf1f3}, substituting in the equation and extracting the terms involving $q^{2n+1}$ and $q^{2n}$ from the resulting equation, we get
			\begin{align}
				\sum_{n= 0}^{\infty } \overline{S_3}(6n+5)q^n\equiv 
				4f_{2}^{}f_6^3\pmod{8}\label{tsp3n2new1}
			\end{align}
			and
			\begin{align}
				\sum_{n= 0}^{\infty } \overline{S_3}(6n+2)q^n\equiv 
				4\frac{ f_{3}^3}{f_1}-2f_4^2-2q\frac{f_{24}^2}{f_2^2f_6^2}\pmod{8}\label{tsp3n2new2}.
			\end{align} 
			Invoking \eqref{f33byf1} in the above equation and extracting the terms involving $q^{2n+1}$ from the resulting equation, we get
			\begin{align}
				\sum_{n= 0}^{\infty } \overline{S_3}(12n+6)q^n\equiv 
				4\frac{ f_{6}^3}{f_2}- 2\frac{f_6^4}{f_1^2 f_{3}^2}\pmod{8}.\label{tsp3n2new3}
			\end{align}
			Using \eqref{tsp3n32186} along with \eqref{tsp3n2new} and \eqref{tsp3n2new3}, we arrive at congruence \eqref{tsp4nx0}. 
		\end{proof}
		
		%  \begin{thm}
			%	If $n$ cannot be represented as the sum of a triangular number and a pentagonal number, then 
			%	\begin{align}\label{t5}
				%		&\overline{S_3}(6n+1) \equiv 0 \pmod{8}.
				%	\end{align}
			%	\end{thm}
		
		%	\begin{proof}
			%	Thanks to \eqref{bt}, equation \eqref{tsp3n1} can be rewritten as 
			%   \begin{align}
				%\overline{S_3}(3n+1)q^n\equiv 2\frac{f_4f_6^2 }{f_2^2 f_{12}}\frac{f_3^3}{f_1} \pmod{8}.\label{tsp3n121}
				%\end{align}
				%Invoking \eqref{f33byf1} and extracting the even powers of $q$ on both sides of the resulting equation, we get
				%\begin{align}
				%\overline{S_3}(6n+1)q^{n}\equiv 2\frac{f_2^4f_6^6 }{f_2^4 f_{12}^2}\equiv2\frac{f_2^4}{f_1^4} \pmod{8},\label{tsp3n122}
				%\end{align}
				%which implies that
				%\begin{align}
				%\overline{S_3}(6n+1)q^{n}\equiv 2\psi(q)f_1. \pmod{8}.\label{tsp3n123}
				%\end{align}
				%This completes the proof of \eqref{t5}.   
				%\end{proof}   
				
				\begin{thm}
					For any positive integer $n$ and $\alpha\geq0$, we have
					\begin{align}
						&\overline{S_3}(2^\alpha(2n))\equiv\overline{S_3}(2n)\pmod{4},\label{res0}\\
						&\overline{S_3}(3^{\alpha}(2n+1)) \equiv \overline{S_3}(2n+1)\pmod{4},\label{t13} \\
						&\overline{S_3}(32n+20) \equiv 0\pmod{4},\label{res22}\\
						&\overline{S_3}(96n+12j+12) \equiv 0\pmod{4} \quad \text{ for } j\in \{{1,2,3,4,5}\},\label{res23}\\
						&\overline{S_3}(3^\alpha(12n+6))\equiv(-1)^{\alpha}\overline{S_3}(12n+6)\pmod{9},\label{resn9}\\
						&\overline{S_3}(3^\alpha(12n+6))\equiv\overline{S_3}(12n+6)\pmod{12},\label{res1}\\
						&\overline{S_3}(144n+24i+18) \equiv0\pmod{12}\,\, \text{for}\,\,i
						\in \{{1, 2, 3, 4, 5}\}\label{s39n41},\\
						&\overline{S_3}(48n+30)\equiv 0\pmod{12},\label{s39n4nr}\\
						&\overline{S_3}(48n+42)\equiv 0\pmod{12},\label{s39n4nc}\\
						%&\overline{S_3}(36n+30) \equiv 0 \pmod{12},\\
						%	&\overline{S_3}(144n+36i+6) \equiv 0 \pmod{12},\,\,\text{for}\,\, j={1, 2,3}.\label{2ntm6}\\
						%   &\overline{S_3}(36n+6) \equiv \begin{cases} 
							%	6 \pmod{12}, & \text{if } n = 2k(3k-1), \\
							%	0 \pmod{12}, & \text{otherwise.}
							%\end{cases}\label{t12}	\\
							&\overline{S_3}(48n+46)\equiv0\pmod{16},\label{mod16}\\
							&\overline{S_3}(24n+22) \equiv 0\pmod{24}.\label{s41}
						\end{align}
					\end{thm}
					\begin{proof}
						Setting $t = 3$ in \eqref{tsp} and using \eqref{f3f1}, we obtain
						\begin{align}
							\sum_{n= 0}^{\infty } \overline{S_3}(n) q^n = 
							\frac{f_4^3 f_{12}^3}{f_2^2f_6^2f_8 f_{24}} 
							+ 2q \frac{f_8 f_{24}}{f_2 f_6},
						\end{align}	
						which implies
						\begin{align}
							&\sum_{n= 0}^{\infty } \overline{S_3}(2n) q^{n} =\frac{f_2^3 f_{6}^3}{f_1^2f_3^2f_4 f_{12}}  \label{s31}
						\end{align}                and
						\begin{align}
							&\sum_{n= 0}^{\infty } \overline{S_3}(2n+1) q^n =2\frac{f_4 f_{12}}{f_1 f_3}.\label{s32}
						\end{align}	
						%	Invoking \eqref{1byf1f3} in \eqref{s32} and extracting the terms involving powers of the form $q^{2n}$ and $q^{2n+1}$ from both sides of the resulting equation, we get
						%	\begin{align}
							%		\sum_{n= 0}^{\infty } \overline{S_3}(4n+1) q^{n} &=2\frac{f_4^2 f_{6}^6}{f_1^2f_3^4f_{12}^2}  \label{s33}\\
							%		\sum_{n= 0}^{\infty }\overline{S_3}(4n+3) q^{n}  &=2\frac{f_2^6 f_{12}^2}{f_1^4 f_3^2f_4^2}.\label{s34}
							%	\end{align}
						Using \eqref{bt}, equation \eqref{s32} reduces to 
						\begin{align}
							\sum_{n= 0}^{\infty } \overline{S_3}(2n+1) q^n =2\frac{f_4 
								f_{3}^3}{{f_1}} \pmod{4}.\label{s323}    
						\end{align}
						Congruence \eqref{t13} follows from above equation and \eqref{ref1sp}.
						Using \eqref{1byf12f32} in \eqref{s31} and extracting the terms containing $q^{2n}$ and $q^{2n+1}$ from both sides of the resulting equation, we get 
						\begin{align}
							\sum_{n= 0}^{\infty }\overline{S_3}(4n) q^{n}  &=\frac{f_4^5 f_{12}^5}{f_1^2f_3^2f_2f_6f_8^2f_{24}^2}+4q\frac{f_2f_6f_8^4f_{24}^2}{f_1^2f_3^2f_4f_{12}}\label{s35n}
						\end{align}
						and
						\begin{align}
							\sum_{n= 0}^{\infty }\overline{S_3}(4n+2) q^{n}  &=2\frac{f_2^3 f_{6}^3}{f_1^3 f_3^3}.\label{s35}
						\end{align}
						Thanks to \eqref{bt}, equation \eqref{s35n} reduces to 
						
						\begin{align}
							\sum_{n= 0}^{\infty }\overline{S_3}(4n) q^{n}  &\equiv\frac{f_4 f_{12}}{f_2f_6f_1^2f_3^2}\pmod{4}.\label{s35n11}
						\end{align}
						Note that
						\begin{align}
							\sum_{n= 0}^{\infty }\overline{S_3}(2n)q^{n}  &=\frac{f_4 f_{12}}{f_2f_6f_1^2f_3^2}\pmod{4}.\label{s35n3}
						\end{align}
						Congruence \eqref{res0} follows from \eqref{s35n11} and \eqref{s35n3}. 
						Substituting \eqref{1byf12f32} in \eqref{s35n11} and extracting terms containing $q^{2n}$ and $q^{2n+1}$, we get
						\begin{align}
							\sum_{n= 0}^{\infty }\overline{S_3}(8n)q^{n}  &\equiv\frac{f_4 f_{12}}{f_2f_6f_1^2f_3^2}\pmod{4}\label{s35n2a}
						\end{align}
						and
						\begin{align}
							\sum_{n= 0}^{\infty }\overline{S_3}(8n+4)q^{n}  &\equiv2\frac{f_2^5 f_{6}^5}{f_1^7f_3^7}\equiv2\frac{f_2^2 f_{6}^2}{f_1f_3}\pmod{4}.\label{s35n2}
						\end{align}
						Invoking \eqref{1byf1f3} in \eqref{s35n2} and extracting the terms involving $q^{2n}$ and $q^{2n+1}$, we get
						\begin{align}
							&\sum_{n= 0}^{\infty } \overline{S_3}(16n+4) q^{n} \equiv 2\frac{f_4^2 f_{6}^5}{f_2f_3^2f_{12}^2} \equiv 2\frac{f_4^2}{f_2}\pmod4. \label{s53}
						\end{align}
						and
						\begin{align}
							&\sum_{n= 0}^{\infty } \overline{S_3}(16n+12) q^{n} \equiv 2\frac{f_2^5 f_{12}^2}{f_1^2f_4^2f_6} \equiv 2f_6^3\pmod4. \label{s54}
						\end{align}
						Congruences \eqref{res22} and \eqref{res23} follows from \eqref{s53} and \eqref{s54}, respectively. 
						Invoking \eqref{f23byf13} in \eqref{s35} extracting the terms containing $q^{3n},$ $q^{3n+1}$ and $q^{3n+2}$ from both sides of the resulting equation, we obtain the following generating functions 	
						\begin{align}
							\sum_{n= 0}^{\infty } \overline{S_3}(12n+2) q^{n} &=2\frac{f_2^4}{f_1^4}+24q\frac{f_2^5f_6^5}{f_1^9f_3}, \label{s36}\\            
							\sum_{n= 0}^{\infty }\overline{S_3}(12n+6) q^{n}  &=6\frac{f_2^7 f_{3}^5}{f_1^{11} f_6},\label{s37}\end{align}
						\begin{align}
							\sum_{n= 0}^{\infty } \overline{S_3}(12n+10) q^{n} &=12\frac{f_2^6 f_3^2 f_6^2}{f_1^{10}}.  \label{s38}
						\end{align}
						Thanks to \eqref{bt} the equation \eqref{s36} reduces to 
						\begin{equation}
							\sum_{n= 0}^{\infty } \overline{S_3}(12n+2) q^{n} \equiv 2\frac{f_2^4}{f_1^4} \equiv 2 f_2^2\pmod{8}.\label{s46}
						\end{equation}
						The following congruence is immediate from the above equation 
						$$\overline{S_3}(24n+14) \equiv 0 \pmod{8}.$$
						Thanks to \eqref{bt} the equation \eqref{s37}, reduces to 
						
						\begin{align}
							\sum_{n= 0}^{\infty }\overline{S_3}(12n+6) q^{n}  &\equiv6\frac{f_2^3 f_{6} f_3}{f_1^{3}}\pmod{8}.\label{s37n1}
						\end{align}
						Invoking \eqref{f3byf13} in the above equation and extracting terms involving powers of $q^{2n}$ on both sides of the resulting equation, we have
						\begin{align}
							\sum_{n= 0}^{\infty }\overline{S_3}(24n+6) q^{n}  &\equiv6\frac{f_2^3 f_{6} f_3}{f_1^{3}}\equiv6\frac{f_2^6f_3^4}{f_1^6f_6^2}\equiv6\psi(q)f_1^3\pmod{8}.\label{s37n2}
						\end{align}
						Thanks to \eqref{bt} equation \eqref{s38}, reduces to
						\begin{align}
							\sum_{n= 0}^{\infty } \overline{S_3}(12n+10) q^{n} &\equiv12\frac{f_2^2 f_6^2 f_3^2}{f_1^{2}}\pmod{16}.  \label{s38ne1}
						\end{align}
						Invoking \eqref{f3.2_f1.2} in the above equation and extracting terms involving powers of $q^{2n+1}$ on both sides of the resulting equation, we have
						\begin{align}
							\sum_{n= 0}^{\infty } \overline{S_3}(24n+22) q^{n} &\equiv8f_2^2 f_6^3\pmod{16}.  \label{s38ne2}
						\end{align}
						Congruence \eqref{mod16} is obtained from above equation.
						Using \eqref{bt} in \eqref{s37}, we have 
						\begin{align}
							\displaystyle \sum_{n= 0}^{\infty }\overline{S_3}(12n+6) q^{n}  &\equiv6\frac{f_2^2f_3^3}{f_1}\pmod{12}.\label{s39}
						\end{align}
						Invoking \eqref{f22byf1} in the above equation and extracting the terms involving powers of the form $q^{3n}$,  $q^{3n+1}$ and $q^{3n+2}$ respectively from both sides of the resulting equation, we get
						\begin{align}
							&\sum_{n= 0}^{\infty } \overline{S_3}(36n+6) q^{n} \equiv 6{f_1^2f_2} \equiv 6\psi(q)f_1\pmod{12},\label{s40}\\
							&\sum_{n= 0}^{\infty } \overline{S_3}(36n+18) q^{n} \equiv 6{f_1^3f_{3}^3} \equiv 6\frac{f_2^2f_3^3}{f_1}\pmod{12},\label{s42}
						\end{align}
						and 
						\begin{equation}
							\overline{S_3}(36n+30)\equiv 0 \pmod{12}.
						\end{equation}
						Congruence \eqref{res1} follows from \eqref{s39} and \eqref{s42}. On the other hand, let us utilize \eqref{f33byf1} in \eqref{s39}, we get 
						\begin{align}
							\displaystyle \sum_{n= 0}^{\infty }\overline{S_3}(12n+6) q^{n}  &\equiv6f_2^2\left[\frac{f_4^3f_6^2}{f_2^2f_{12}}+q\frac{f_{12}^3}{f_4}\right]\pmod{12}.\label{s39n}
						\end{align}
						From the above equation \eqref{s39n}, we get 
						\begin{align}
							\displaystyle \sum_{n= 0}^{\infty }\overline{S_3}(24n+6) q^{n}  &\equiv6f_2^3\pmod{12}\label{s39n3}
						\end{align}
						and
						\begin{align}
							\displaystyle \sum_{n= 0}^{\infty }\overline{S_3}(24n+18) q^{n}  &\equiv6f_6^3\pmod{12}\label{s39n4},
						\end{align}        
						for $\beta\geq1$, using equations \eqref{s39n3} and \eqref{s39n4}, we have
						\begin{align}
							\overline{S_3}(48n+6)\equiv \overline{S_3}(3^{\beta}(48n+6))\pmod{12}\label{s39n4n}.
						\end{align}
						By extracting the terms involving powers of $q^{2n+1}$  from \eqref{s39n3} and \eqref{s39n4}, we obtain the congruences \eqref{s39n4nr} and \eqref{s39n4nc}.
						Further, by collecting the terms containing powers of the form $q^{6n+i},$ where $i=1,2,..., 5$, from both sides of \eqref{s39n4}, we arrive at the congruence \eqref{s39n41}. Thanks to \eqref{bt}, equation \eqref{s38} reduces to
						\begin{align}
							\sum_{n= 0}^{\infty } \overline{S_3}(12n+10) q^{n} \equiv12\frac{f_2^6 f_3^2 f_6^2}{f_1^{10}} \equiv 12 f_2f_6^3\pmod{24}.  \label{s43}
						\end{align} 
						By extracting the terms involving powers of $q^{2n+1}$  from \eqref{s43}, we arrive at \eqref{s41}.	Further, equation \eqref{s43} yields
						\begin{align}
							\sum_{n= 0}^{\infty } \overline{S_3}(24n+10) q^{n} \equiv 12f_1f_3^3 \equiv12\frac{f_2f_3^3}{f_1}\pmod{24}\label{s44}
						\end{align}
						Invoking \eqref{f33byf1} in \eqref{s44} and extracting the terms involving powers of the form $q^{2n}$ from both sides of the resulting equation, we get
						\begin{align}
							&\sum_{n= 0}^{\infty } \overline{S_3}(48n+10) q^{n} \equiv 12\frac{f_2^3}{f_1} \equiv12\psi(q)f_2\pmod{24}.\label{s45}
						\end{align}    
						%{and}
						%\begin{align}
						%&\sum_{n= 0}^{\infty } \overline{S_3}(48n+34) q^{n} \equiv 12\frac{f_1f_6^3}{f_2} \pmod{24}.\label{eq24}
						%\end{align}
						Using \eqref{f33byf1} in the equation \eqref{s323} and extracting the terms containing $q^{2n}$ and $q^{2n+1}$ respectively, we get
						\begin{align}
							\sum_{n= 0}^{\infty }\overline{S_3}(4n+1)q^n\equiv 2\frac{f_2^4f_3^2}{f_1^2f_{6}} \equiv 2f_4f_2\pmod{4}\label{s60}
						\end{align}
						and
						\begin{align}
							\sum_{n= 0}^{\infty }\overline{S_3}(4n+3)q^n\equiv 2{f_6^3} \pmod{4}.\label{s61}
						\end{align}
						%Congruence \eqref{t8} follows from \eqref{s61}. 
						Again, alternatively invoking \eqref{f4byf1} in \eqref{s32} and extracting the terms involving the powers of $q^n$ and $q^{3n+1}$ respectively from both sides of the resulting equation, we get
						\begin{align}
							\sum_{n= 0}^{\infty } \overline{S_3}(6n+1)q^n\equiv2\frac{f_4 f_6^4}{f_{12}^2} \equiv 2f_4\pmod{4} \label{tsp3n3213}
						\end{align}
						and
						\begin{align}
							\sum_{n= 0}^{\infty } \overline{S_3}(6n+3)q^n\equiv 2\frac{f_2^2 f_{3}^3 f_{12}}{f_1 f_{6}^2} \equiv2\frac{f_4f_3^3}{f_1}\pmod{4},\label{tsp3n3214}
						\end{align}
						%\begin{align}
						%&\overline{S_3}(6n+5)\equiv 0 \pmod{4}.\label{tsp3n3215}
						%\end{align}
						Equation \eqref{tsp3n3213} implies 
						\begin{align}
							\sum_{n= 0}^{\infty } \overline{S_3}(24n+1)\equiv2f_{1}  \pmod{4}.\label{tsp3n32156}
						\end{align}
						and 
						\begin{align}
							\overline{S_3}(24n+6j+1)\equiv0  \pmod{4}\,\, \text{for} \,\,j=\{1,2,3\}.\label{tsp3n3217}
						\end{align}
						Thanks to \eqref{bt}, \eqref{s37} takes the form
						\begin{align}
							\sum_{n= 0}^{\infty }\overline{S_3}(12n+6) q^{n}  &\equiv6\frac{f_2^7 f_{3}^5}{f_1^{11} f_6}\equiv6 f_1f_2f_3f_6\pmod{9}.\label{s371}
						\end{align}
						Invoking \eqref{f1f2} in the above equation and extracting the terms involving $q^{3n+1}$ from both sides of the resulting equation, we obtain
						\begin{align}
							\sum_{n= 0}^{\infty }\overline{S_3}(36n+18) q^{n}  &\equiv-6 f_1f_2f_3f_6\pmod{9}.\label{s3711}
						\end{align}
						Congruence \eqref{resn9} follows from \eqref{s371} and \eqref{s3711}. 
					\end{proof}
					\begin{thm} For any nonnegative $n$, we have
						\begin{align}
							\overline{S_3^2}(2n+1)=2\cdot\overline{S_3}(4n+2).\label{t110}
						\end{align}
					\end{thm}
					\begin{proof}
						Consider \eqref{tspr} with $r=2$ and $t=3$, we have
						\begin{equation}\label{tspr1}
							\sum_{n \geq 0} \overline{S_3^2}(n) q^n = \frac{f_{3}^{4}f_2^{6}  f_{12}^2}{ f_1^{4}f_4^2 f_{6}^{6}}.
						\end{equation}
						Invoking \eqref{f34f14} in \eqref{tspr1} and extracting the terms containing $q^{2n+1}$ from both sides of the resulting equation, we get
						\begin{equation}\label{tspr2}
							\sum_{n \geq 0} \overline{S_3^2}(2n+1) q^n = 4\frac{f_2^{3} f_{6}^{3}}{f_1^{3} f_3^3}.
						\end{equation}
						Relation \eqref{t110} follows from \eqref{s35} and \eqref{tspr2}. 
					\end{proof}
					%p-dissections
					\begin{thm}\label{reff}
						For any prime $p\equiv5\pmod{6},\alpha\geq 0\hspace{1mm} \text{and} \hspace{1mm}n\ge0$, we have 
						\begin{equation}
							\overline{S}_3^{16m + 1}\left ( 24p^{2\alpha +2}n+\left (24i+7 p\right)p^{2\alpha +1}+4 \right )\equiv0\pmod{32} ,\label{d41}
						\end{equation}
						where $i$ is an integer and $1\leq i\leq p-1$.
					\end{thm}
					\begin{proof}  
						Using \eqref{mainth} with $\zeta=4$, \( \eta= 1 \) and $t=3$, thanks to \eqref{bt} the equation \eqref{eq3} takes the form
						\begin{equation}
							\sum_{n= 0}^{\infty }\overline{S}_3^{16m + 1}(12n+11)q^n\equiv 16\frac{f_2^3f_3^3}{f_1} \pmod{32}.\label{s55}	
						\end{equation}
						Using \eqref{f33byf1} in \eqref{s55} and extracting the terms containing $q^{2n}$ on both sides of the resulting equation, we have
						\begin{align}
							\sum_{n= 0}^{\infty }\overline{S}_3^{16m + 1}(24n+11)q^n&\equiv 16\frac{f_1f_2^3f_3^2}{f_6}\equiv16\frac{f_2^4}{f_1}\equiv 16\psi(q)f_{4} \pmod{32}.\label{d42}
						\end{align}		
						Define
						\begin{equation}
							\sum_{n= 0}^{\infty }d(n)q^{n}=\psi(q)f_{4}.\label{d43}
						\end{equation}
						Combining \eqref{d42} and  \eqref{d43}, we see that 
						\begin{equation}
							\overline{S}_3^{16m + 1}(24n+11)\equiv16d(n)\pmod {32}.\label{d44}
						\end{equation}
						Now, we consider the congruence equation 
						\begin{equation}
							\frac{k^{2}+k}{2}+4\cdot  \frac{3m^{2}+m}{2}\equiv \frac{7p^{2}-7}{24}\pmod{p}	,\label{d45}
						\end{equation}
						\begin{equation}
							3\cdot (2k+1)^{2}+4\cdot (6m+1)^{2}\equiv 0\pmod{p},
						\end{equation}
						where $\frac{-(p-1)}{2}\le m \le \frac{p-1}{2}$ and $p$ is a prime such that $\left(\frac{-12}{p}\right)=-1$. Since $\left(\frac{-12}{p}\right)=-1$ for $p\equiv 5\pmod{6}$, the congruence relation \eqref{d45} holds if and only if $m=\dfrac{\pm p-1}{6}$ and $k=\dfrac{p-1}{2}$. Therefore, if we  substitute \eqref{a17} and  \eqref{a2} into \eqref{d43}  and then extract the terms in which the powers of $q$ are $pn+\frac{7p^2-7}{24}$, we arrive at
						\begin{equation}
							\sum\limits_{n=0}^{\infty}d\left(pn+\frac{7p^2-7}{24}\right)q^{pn+\frac{7p^2-7}{24}}=(-1)^{\frac{\pm p-1}{6}} q^{\frac{7p^2-7}{24}} \psi(q^{p^{2}})f_{4p^{2}} .\label{d46}
						\end{equation}
						Dividing $q^{\frac{7p^2-7}{24}}$ on both sides of \eqref{d46} and then replacing $q^p$ by $q$, we find that
						\begin{equation}
							\sum\limits_{n=0}^{\infty}d\left(pn+\frac{7p^2-7}{24}\right)q^n=(-1)^{\frac{\pm p-1}{6}}\psi(q^{p})f_{4p},\label{d47}
						\end{equation}
						which implies that
						\begin{equation}
							\sum\limits_{n=0}^{\infty}d\left(p^2n+\frac{7p^2-7}{24}\right)q^n=(-1)^{\frac{\pm p-1}{6}}\psi(q)f_{4},\label{d48}
						\end{equation}
						and for $n\ge0$,
						\begin{equation}
							d\left(p^2n+pi+\frac{7p^2-7}{24}\right)=0,\label{d49}
						\end{equation}
						where $i$ is an integer and $1\le i\le p-1$.\\
						Combining \eqref{d43} and \eqref{d48}, we see that for $n\ge0$,
						\begin{equation}
							d\left(p^2n+\frac{7p^2-7}{24}\right)=(-1)^{\frac{\pm p-1}{6}}d(n).\label{d50}
						\end{equation}
						By \eqref{d50} and mathematical induction, we deduce that for $n\ge0$ and $\alpha \ge 0$,
						\begin{equation}
							d\left(p^{2\alpha}n+\frac{7p^{2\alpha}-7}{24}\right)=(-1)^{\alpha.\frac{\pm p-1}{6}}d(n).\label{d51}
						\end{equation}
						Replacing $n$ by $p^2n+pi+\frac{7p^2-7}{24}$ $(i=1, 2,\ldots, p-1)$ in \eqref{d51} and using \eqref{d49}, we deduce that for $n\ge 0$ and $\alpha \ge 0$,
						\begin{equation*}
							d\left(p^{2\alpha+2}n+p^{2\alpha+1}i+\frac{7p^{2\alpha+2}-7}{24}\right)=0.
						\end{equation*}
						Again, replacing $n$ by $p^{2\alpha+2}n+p^{2\alpha+1}i+\frac{7p^{2\alpha+2}-7}{24}$ $(i=1, 2,\ldots, p-1)$ in \eqref{d44}, we arrive at \eqref{d41}.
					\end{proof}
					
					\begin{thm}
						For any prime $p\equiv5\pmod{6},\alpha\geq 0\hspace{1mm} \text{and} \hspace{1mm}n\ge0$, we have 
						\begin{equation}
							\overline{S}_3\left ( 24p^{2\alpha +2}n+\left (24i+7 p\right)p^{2\alpha +1}\right )\equiv0\pmod{32}  ,\label{d141p1}
						\end{equation}
						\begin{equation}\label{d141p}
							\overline{S}_3^{16m + 1}\left ( 24p^{2\alpha +2}n+\left (24i+19 p\right)p^{2\alpha +1}\right )\equiv0\pmod{32} ,
						\end{equation}
						where $i$ is an integer and $1\leq i\leq p-1$.
					\end{thm}
					\begin{proof}  
						%Using \eqref{mainth} with $\zeta=4$, \( \eta= 1 \) and $t=3$, 
						Thanks to \eqref{bt} the equation \eqref{eq2} takes the form
						\begin{equation}
							\sum_{n= 0}^{\infty }\overline{S}_3(12n+7)q^n\equiv 8\frac{f_8^2}{f_1^2} \pmod{32}.\label{s1551}	
						\end{equation}
						Using \eqref{1byf12} in \eqref{s1551} and extracting the terms containing $q^{2n}$ and $q^{2n+1}$ on both sides of the resulting equation, we have
						\begin{equation}
							\sum_{n= 0}^{\infty }\overline{S}_3(24n+7)q^n\equiv 8\frac{f_2^2f_4}{f_1}\equiv8\psi(q)f_{4}\pmod{32}\label{s155a}	
						\end{equation}
						and
						\begin{equation}
							\sum_{n= 0}^{\infty }\overline{S}_3(24n+19)q^n\equiv 16\frac{f_4f_8^2}{f_1}\equiv16\psi(q)f_{16}\pmod{32}.\label{s155}	
						\end{equation}
						We utilize \eqref{s155a}, to prove \eqref{d141p1}. Using \eqref{mainth} with $\zeta=4$, \( \eta= 1 \) and $t=3$, along with \eqref{s155}, we arrive at \eqref{d141p}. 
					\end{proof}
					\begin{thm}
						Let $p\ge 5$ be a prime and $\left(\frac{-6}{p}\right)=-1$. Then for all $\alpha\ge 0, n\ge 0$, we have 	\begin{equation}
							\overline{S}_{3}\left ( 48p^{2\alpha +2}n+2\left (24i+5p\right)p^{2\alpha +1}\right )\equiv0\pmod{24}	 ,\label{c4121}
						\end{equation}
						where $i$ is an integer and $1\leq i\leq p-1$.
					\end{thm}
					\begin{proof}  
						The proof is similar to the proof of the Theorem \ref{reff}, where \eqref{s45} is used.
					\end{proof}
					\begin{thm}
						Let $p\ge 5$ be a prime and $\left(\frac{-18}{p}\right)=-1,\alpha\geq 0\hspace{1mm} \text{and} \hspace{1mm}n\ge0$, we have 
						\begin{equation}
							\overline{S}_3^{8m + 1}\left ( 48p^{2\alpha +2} n+2\left (24i+ 11p\right)p^{2\alpha +1} \right )\equiv0\pmod{16}
						\end{equation}
						where $i$ is an integer and $1\leq i\leq p-1$.
					\end{thm}
					\begin{proof}  
						Equation \eqref{s38ne2} can be rewritten as
						\begin{align}
							\sum_{n= 0}^{\infty } \overline{S_3}(48n+22) q^{n} &\equiv8f_2 f_3^3\pmod{16}.\label{s39n30x16}
						\end{align}
						Remaining part of the proof follows the proof of Theorem \ref{reff}, where \eqref{s39n30x16} is used.
					\end{proof}
					\begin{thm}
						For any prime $p\equiv5\pmod{6}$ and $\left(\frac{-3}{p}\right)=-1$. Then for all $\alpha\ge 0, n\ge 0$, we have 	\begin{equation}
							\overline{S}_{3}\left ( 36p^{2\alpha +2}n+6\left (6i+p\right)p^{2\alpha +1}\right )\equiv0\pmod{12}	 ,\label{c412}
						\end{equation}
						where $i$ is an integer and $1\leq i\leq p-1$.
					\end{thm}
					\begin{proof}  
						The proof is similar to the proof of the Theorem \ref{reff}, where \eqref{s40} is used.
					\end{proof}
					%\begin{thm}
					%Let $p\ge 5$ be a prime and $\left(\frac{-2}{p}\right)=-1$. Then for all $\alpha\ge 0, n\ge 0$, we have 
					%\begin{equation}
					%	\overline{S}_{3}\left ( 144p^{2\alpha +2}n+\left 144(i+3p\right)p^{2\alpha +1}\right )\equiv0\pmod{12}	 ,\label{c412x}
					%	\end{equation}
				%	where $i$ is an integer and $1\leq i\leq p-1$.
				%\end{thm}
				%\begin{proof}  
				%Thanks to \eqref{s39n4}, we have
				%       \begin{align}
					%		\displaystyle \sum_{n= 0}^{\infty }\overline{S_3}(144n+18) q^{n}  &\equiv6f_1f_2\pmod{12}\label{s39n4x}.
					%	\end{align}        
				%Remaining part of the proof follows the proof of Theorem \ref{reff}, where \eqref{s39n4x} is used.
				%\end{proof}
				\begin{thm}
					Let $p\ge 5$ be a prime and $\left(\frac{-1}{p}\right)=-1$. Then for all $\alpha\ge 0, n\ge 0$, we have 	\begin{equation}
						\overline{S}_{3}\left ( 48p^{2\alpha +2}n+6\left (8i+p\right)p^{2\alpha +1}\right )\equiv0\pmod{12}	 ,\label{c4121a}
					\end{equation}
					where $i$ is an integer and $1\leq i\leq p-1$.
				\end{thm}
				\begin{proof}  
					Thanks to \eqref{bt}, \eqref{s39n3} can be rewritten as
					\begin{align}
						\displaystyle \sum_{n= 0}^{\infty }\overline{S_3}(48n+6) q^{n}  &\equiv6f_1^3\pmod{12}.\label{s39n30x1}
					\end{align}
					Remaining part of the proof follows the proof of Theorem \ref{reff}, where \eqref{s39n30x1} is used.
				\end{proof}
				%\begin{thm}
				%Let $p\ge 5$ be a prime and $\left(\frac{-3}{p}\right)=-1$. Then for all $\alpha\ge 0, n\ge 0$, we have 	\begin{equation}
					%	\overline{S}_{3}\left (144p^{2\alpha +2}n+18\left (8i+p\right)p^{2\alpha +1}\right )\equiv0\pmod{12}	 ,\label{c4121}
					%	\end{equation}
				%	where $i$ is an integer and $1\leq i\leq p-1$.
				%\end{thm}
				%\begin{proof}  
				%Equation \eqref{s39n4} can we rewritten as
				%\begin{align}
				%			\displaystyle \sum_{n= 0}^{\infty }\overline{S_3}(144n+18) q^{n}  &\equiv6f_1^3\pmod{12}\label{s39n30x}
				%		\end{align}
			%Remaining part of the proof follows the proof of Theorem \ref{b41}, where \eqref{s39n30x} is used.
			%\end{proof}
			\begin{thm}
				For any prime $p\equiv5\pmod{6},\alpha\geq 0\hspace{1mm} \text{and} \hspace{1mm}n\ge0$, we have 
				\begin{equation}
					\overline{S}_3^{4m + 1}\left ( 6p^{2\alpha +2}n+\left (6i+ p\right)p^{2\alpha +1} \right )\equiv0\pmod{8} ,\label{a41}
				\end{equation}
				where $i$ is an integer and $1\leq i\leq p-1$.
			\end{thm}
			\begin{proof}  
				The proof is similar to the proof of the Theorem \ref{reff}, where \eqref{tsp3n32182} is used.
			\end{proof}
		\begin{thm}
			Let $p\ge 5$ be a prime with $\left(\frac{-1}{p}\right)=-1$. For all$\alpha, n\geq 0$,
			\begin{equation}
				\overline{S}_3^{4m + 1}\left ( 24p^{2\alpha +2} n+6\left (4i+ p\right)p^{2\alpha +1} \right )\equiv0\pmod{8} ,\label{}
			\end{equation}
			for each  $i\in\{1,\dots,p-1\}$.
		\end{thm}
		\begin{proof}  
			The proof is similar to the proof of the Theorem \ref{reff}, where \eqref{s37n2} is used.
		\end{proof}

\section{Acknowledgment}
The authors thank Prof. Mohammed L. Nadji for sharing a copy of his research article \cite{nadji2021}.

{}

\begin{thebibliography}{}			
\bibitem{adiga2018congruences} Adiga, C., Naika, M. M., Ranganatha, D. and Shivashankar, C.:  Congruences modulo 8 for $(2, k)$-regular overpartitions for odd $k>1$. Arab. J. Math. (Springer), 7(2), 61--76, 2018.

\bibitem{baruah2015partitions}
Baruah, N.D. and Ojah, K.K.: 2015. Partitions With Designated Summands in Which All Parts Are Odd. Integers, 15 (A9), p. 16.

\bibitem{berndt2012ramanujan}
Berndt, B. C.: Ramanujan’s notebooks: Part III. Springer Science \& Business Media, (2012).

\bibitem{cui2013}
Cui. Su-Ping, and Nancy SS Gu.: Arithmetic properties of $l$-regular partitions. Adv. Appl. Math., 51(4), 507--523. (2013).

\bibitem{2ABH} Dou, D.Q.J. and Lin, B.L.S.: New Ramanujan-type congruences modulo 5 for overpartitions. Ramanujan J., 44 (2017), 401--410.

\bibitem{mdb}
Hirschhorn, M. D.: The Power of $q$, a personal journey, Developments in Mathematics, v. 49, Springer, 2017.

\bibitem{2ABH1}
Hirschhorn, M.D. and Sellers, J.A.: Arithmetic relations for overpartitions. J. Combin. Math. Combin. Comput.,  53 (2005), 65--73.

\bibitem{2ABH2}
Hirschhorn, M.D. and Sellers, J.A.: An infinite family of overpartition congruences modulo 12. Integers, 5 (2005), A20.

\bibitem{hirschhorn2010elementary}
Hirschhorn, M. D. and  Sellers, J. A.: Elementary proofs of parity results for 5-regular partitions. Bull. Aust. Math. Soc., 81(1),58–-63, (2010).

\bibitem{hirschhorn1993cubic}
Hirschhorn, M.D., Garvan, F. and Borwein, J.: Cubic analogs of the Jacobin cubic theta function $\theta(z,q)$, Canad. J. Math., 45, 673-694, (1993).

\bibitem{hirschhorn2014congruence}
Hirschhorn, M.D. and Sellers, J.A.: A Congruence Modulo 3 for Partitions into Distinct Non-Multiples of Four. J. Integer Seq., 17(9), 14-9, (2014).

\bibitem{2ABE}
Kim, B.: The overpartition function modulo128. Integers, 8 (2008),  A38.

\bibitem{2KB}
Kim, B.: A short note on the overpartition function. Discrete Math., 309 (2009), 2528--2532.

\bibitem{nadji2021}
Nadji, M.L. and Ahmia, M.: Partitions into parts simultaneously regular and distinct. J. Ramanujan Math. Soc. (to appear).

\bibitem{shivashankar2022congruences}
Shivashankar, C and Gireesh, D.S.: Congruences for $\ell$-regular overpartitions into odd parts, Bol. Soc. Mat. Mex. (3), 28 (1), 23, (2022).

\bibitem{toh2012ramanujan}
Toh, P.C.: Ramanujan type identities and congruences for partition pairs. Discrete Math., 312(6), 1244--1250, (2012).

\bibitem{xia2013analogues}
Xia, E.X. and Yao, O.M.: Analogues of Ramanujan’s partition identities, Ramanujan J., 31(3), 373--396, (2013).

\bibitem{xia2012some}
Xia, E.X. and Yao, O.M.: Some modular relations for the G{\"o}llnitz--Gordon functions by an even--odd method, Journal of Mathematical Analysis and Applications, 387(1), 126--138, (2012).

\bibitem{LW}
Wang, L: Arithmetic properties of $7$-regular partition, Ramanujan J., 47, 99--115, (2018).
\end{thebibliography}
\end{document}